\documentclass{amsart}
\title[lattices in solvable Lie groups]
{Singularity of the varieties of  representations of lattices in solvable Lie groups}

\author{Hisashi Kasuya}

\usepackage{amssymb}
\usepackage{amsmath}
\usepackage{amscd}
\usepackage{amstext}
\usepackage{amsfonts}
\usepackage[all]{xy}

\theoremstyle{plain}

\theoremstyle{plain}

\theoremstyle{plain}

\theoremstyle{plain}
\newtheorem{theorem}{Theorem}[section] 
\theoremstyle{remark}

\theoremstyle{Main result}
\newtheorem{main result}{Main result}
\theoremstyle{lemma}
\newtheorem{lemma}[theorem]{Lemma}
\theoremstyle{definition}
\newtheorem{definition}[theorem]{Definition}
\theoremstyle{proposition}
\newtheorem{proposition}[theorem]{Proposition}
\theoremstyle{corollary}
\newtheorem{corollary}[theorem]{Corollary}
\theoremstyle{remark}
\newtheorem{example}{Example}

\address[Hisashi Kasuya]{Department of Mathematics, Tokyo Institute of Technology, 1-12- 1, O-okayama, Meguro, Tokyo 152-8551, Japan}
\email{kasuya@math.titech.ac.jp}


\newcommand{\C}{\mathbb{C}}
\newcommand{\R}{\mathbb{R}}

\newcommand{\g}{\frak{g}}
\newcommand{\n}{\frak{n}}

\begin{document} 

\maketitle

\begin{abstract}
For a lattice $\Gamma$ of a simply connected solvable Lie group $G$,
we describe the analytic germ  in the  variety of representations of $\Gamma$  at the trivial representation  
as an analytic germ which is linearly embedded in the analytic germ associated with the nilpotent Lie algebra determined by $G$.
By this description, under certain assumption, we study the singularity of the analytic germ  in the  variety of representations of $\Gamma$  at the trivial representation by using the Kuranishi space construction.
By a similar technique, we also study deformations of holomorphic  structures of trivial vector bundles over complex parallelizable solvmanifolds.
\end{abstract}
\section{introduction}


Let $X$ be an analytic germ in $\C^{n}$ at the origin defined by analytic equations
\[f_{1}(z)=0,\dots,f_{k}(z)=0.
\]
We say that $X$ is cut out  by polynomial equations of degree at most $\nu$ if \[f_{1}(z),\dots,f_{k}(z)\] are polynomial functions of degree at most $\nu$ with trivial linear terms.
We say that an analytic germ $Y$ is linearly embedded in $X$ if for a  subspace  $V\subset \C^{n}$,  the germ  $Y$ is equivalent to an analytic germ in $V$  at the origin defined by analytic equations
\[f_{1}(z)=0,\dots,f_{k}(z)=0, \,\,\, z\in V.
\]
If $X$ is cut out  by polynomial equations of degree at most $\nu$ and $Y$ is linearly embedded in $X$, then $Y$ is also
cut out  by polynomial equations of degree at most $\nu$.
 
Let $\Gamma$ be a finitely generated group,  $A$ a linear algebraic group with  Lie algebra $\frak a$ and $R(\Gamma, A)$ the  set of homomorphisms $\Gamma\to A$.
Then $R(\Gamma, A)$ can be considered as an affine algebraic variety.
For a representation $\rho\in R(\Gamma, A)$ we are interested    in the analytic germ $(R(\Gamma, A),\rho)$. 
The singularity of the analytic germ $(R(\Gamma, A),\rho)$ can be considered as an obstruction of deformations of $\rho$.

If $\Gamma$ is the fundamental group of a manifold $M$,
we can geometrically describe 
the analytic germ $(R(\Gamma, A),\rho)$  by using the deformation theory of differential graded Lie algebras (for short, DGLAs) developed by Goldman and Millson \cite{GM1}, \cite{GM2}.
By such technique and the Hodge theory of local systems over  K\"ahler manifolds studied by Simpson \cite{Sim}, if $\Gamma$ is a K\"ahler group (i.e. a group which can be  the fundamental group of a compact K\"ahler manifold), then for a semisimple representation $\rho\to {\rm GL}_{m}(\C)$, the analytic germ $(R(\Gamma, {\rm GL}_{m}(\C)),\rho)$ is
cut out  by polynomial equations of degree at most $2$.

However in general the analytic germ $(R(\Gamma, A),\rho)$ is not cut out  by polynomial equations of degree at most $2$.
In \cite{GM1}, Goldman and Millson  observed that for a lattice $\Gamma$ in the three dimensional real Heisenberg group, the analytic germ $(R(\Gamma, A),{\bf 1})$ at the trivial representation ${\bf 1}$ is equivalent to a cubic cone.
In this paper we consider a certain  class of groups which contains this example.

Let $\Gamma$ be a lattice in a simply connected solvable Lie group $G$.
Then the solvmanifold $G/\Gamma$ is an aspherical manifold with the fundamental group $\Gamma$.
The purpose of this paper is to study the analytic germ $(R(\Gamma, A),{\bf 1})$ at  the trivial representation ${\bf 1}$.

For a manifold $M$ 
 the analytic germ at the trivial representation of the fundamental group of $M$ can be studied by the differential graded algebra (for short, DGA) $A^{\ast}(M)$ of  differential forms on  $M$.
 The following result is known.
\begin{theorem}[\cite{GM1,GM2},\cite{DP}]\label{DP}
Let $M$ be a compact manifold with the fundamental group $\Gamma$.
Suppose that we have a finite dimensional sub-DGA $C^{\ast}\subset A^{\ast}(M)\otimes \C$ such that the inclusion induces a cohomology isomorphism and $C^{0}=\C$.
Then  the analytic germ $(R(\Gamma, A),{\bf 1})$ at the trivial representation ${\bf 1}$ is equivalent to  the analytic germ $(F(C^{\ast}, \frak a), 0)$ at the origin $0$ for the affine variety
\[F(C^{\ast}, \frak a)=\left\{\omega\in C^{\ast}\otimes \frak a: d\omega+\frac{1}{2}[\omega,\omega]=0\right\}.
\] 
\end{theorem}
Consider a solvmanifold $G/\Gamma$, Lie algebra $\g$ of $G$ and the cochain complex $\bigwedge \g^{\ast}$ which is regarded as a differential graded algebra of left-$G$-invariant forms on $G/\Gamma$.
Suppose that $G$ is completely solvable.
In \cite{Hatt}  Hattori proved that  the inclusion  $\bigwedge \g^{\ast}\subset A^{\ast}(G/\Gamma)$ induces a cohomology isomorphism.
By  Theorem \ref{DP} and Hattori's theorem, in \cite{DP}, Dimca and Papadima remarked that the analytic germ $(R(\Gamma, A),{\bf 1})$ is equivalent to 
 the analytic germ $(F(\bigwedge \g^{\ast}, \frak a), 0)$ at the origin $0$.
However, for a general solvmanifold $G/\Gamma$, the inclusion  $\bigwedge \g^{\ast}\subset A^{\ast}(G/\Gamma)$ does not induce a cohomology isomorphism.

In this paper, we consider general solvmanifolds.
Let $\g$ be a solvable Lie algebra.
Then we can define the nilpotent Lie algebra $\frak u$ called nilshadow of $\g$ which is uniquely determined by $\g$, as shown in \cite{DER}.

\begin{theorem}[\cite{K2}]
Let $G$ be a simply connected solvable Lie group with a lattice $\Gamma$ and $\g$ the Lie algebra of $G$.
We consider the nilshadow $\frak u$ of $\g$.
Then we have a  sub-DGA   $A^{\ast}_{\Gamma}\subset A^{\ast}(G/\Gamma)\otimes\C$ such that:\begin{itemize}
\item The inclusion $A^{\ast}_{\Gamma}\subset A^{\ast}(G/\Gamma)\otimes\C$ induces an  isomorphism in cohomology.
\item $A^{\ast}_{\Gamma}$ can be regarded as a sub-DGA of $\bigwedge \frak u^{\ast}\otimes \C$.
\end{itemize}
\end{theorem}
See Section \ref{2SE} for the constructions of the nilshadow and DGA $A_{\Gamma}^{\ast}$.
By this theorem and Theorem \ref{DP}, the analytic germ
$(R(\Gamma, A),{\bf 1})$ is equivalent to
 the analytic germ $(F(A^{\ast}_{\Gamma}, \frak a), 0)$.
By the second assertion of the theorem we have the following theorem.

\begin{theorem}\label{UTD}
Let $\Gamma$ be a lattice in a simply connected solvable Lie group $G$ and $\g$ the Lie algebra of $G$.
Let  $A$ be a linear algebraic group with the Lie algebra $\frak a$.
Consider the nilshadow $\frak u$ of $\g$.
Then the analytic germ $(R(\Gamma, A),{\bf 1})$ at the trivial representation ${\bf 1}$ is  linearly embedded in the analytic germ $(F(\bigwedge\frak u^{\ast}, \frak a), 0)$ at the origin $0$ for the affine variety
\[F(\bigwedge\frak u^{\ast}, \frak a)=\left\{\omega\in\bigwedge \frak u^{\ast}\otimes \frak a: d\omega+\frac{1}{2}[\omega,\omega]=0\right\}
\]
\end{theorem}

This theorem is useful for estimating the singularity of the analytic germ $(R(\Gamma, A),{\bf 1})$.

Let $\n$ be a $\nu$-step nilpotent Lie algebra.
Consider the lower central series
\[\n=\n^{(1)}\supset \n^{(2)}\supset\dots \supset \n^{(\nu)}(\not=\{0\})\supset  \n^{(\nu+1)}=\{0\}
\]
where $\n^{(i+1)}=[\n,\n^{(i)}]$.
Take a subspace $\frak a^{(i)}$ such that $\n^{(i)}=\n^{(i+1)}\oplus \frak a^{(i)}$.
We have 
\[\n=\frak a^{(1)}\oplus \frak a^{(2)}\oplus \dots \oplus  \frak a^{(\nu)}.
\]
It is known that $[\n^{(i)},\n^{(j)}]\subset \n^{(i+j)}$ (see \cite{Cor}).
A nilpotent Lie algebra $\frak n$ is called naturally graded if
we can choose subspaces $\frak a^{(i)}$ such that $[\frak a^{(i)},\frak a^{(j)}]\subset \frak a^{(i+j)}$.

We prove the following proposition by using the construction of Kuranishi spaces of DGLAs.
\begin{proposition}\label{KKKUU}
Let $\n$ be a $\nu$-step naturally graded nilpotent Lie algebra and $\frak g$ a Lie algebra.
Then  the analytic germ $(F(\bigwedge \frak n^{\ast}, \frak g), 0)$
 is cut out by polynomial equations of degree at most $\nu+1$.
\end{proposition}

By this proposition and Theorem \ref{UTD}, we have the following theorem.

\begin{theorem}\label{realnat}
Let $\Gamma$ be a lattice in a simply connected solvable Lie group $G$ and $\g$ the Lie algebra of $G$.
Let  $A$ be a linear algebraic group with a Lie algebra $\frak a$.
Consider the nilshadow $\frak u$ of $\g$.
We suppose that the Lie algebra $\frak u$  is $\nu$-step naturally graded.
Then
the analytic germ $(R(\Gamma, A),{\bf 1})$ at the trivial representation ${\bf 1}$   is  cut out by
 polynomial equations of degree 
at most $\nu+1$.
\end{theorem}

Let $\frak n$ be a two-step nilpotent Lie algebra.
For any complement $\frak a^{(1)}$ of $\n^{(2)}$ in $\frak n$, we have $\n=\frak a^{(1)}\oplus \n^{(2)}$ and $[\frak a^{(1)},\frak a^{(1)}]\subset \n^{(2)}$
and so a two-step nilpotent Lie algebra $\frak n$ is naturally graded.
Hence as an application of Theorem \ref{realnat}, we have the following Corollary
\begin{corollary}
Let $\Gamma$ be a lattice in a simply connected solvable Lie group $G$ and $\g$ the Lie algebra of $G$.
Let  $A$ be a linear algebraic group with a Lie algebra $\frak a$.
Consider the nilshadow $\frak u$ of $\g$.
We suppose that the Lie algebra $\frak u$  is two-step nilpotent.
Then
the analytic germ $(R(\Gamma, A),{\bf 1})$ at the trivial representation ${\bf 1}$ is    cut out by
 polynomial equations of degree 
at most $3$.
\end{corollary}

%

\section{Nilshadows and cohomology of solvmanifolds}\label{2SE}
Let $\g$ be a solvable $K$-Lie algebra for $K=\R$ or $\C$.
Let $\n$ be the nilradical of $\g$.
There exists a subvector space (not necessarily Lie algebra) $V$ of $\g$ so that
$\g=V\oplus \n$ as the direct sum of vector spaces and for any  $A,B\in V$ $({\rm ad}_A)_{s}(B)=0$ where $({\rm ad}_A)_{s}$  is the semi-simple part of ${\rm ad}_{A}$ (see \cite[Proposition I\hspace{-.1em}I\hspace{-.1em}I.1.1] {DER}).
We define the  map ${\rm ad}_{s}:\g\to D(\g)$ as 
${\rm ad}_{sA+X}=({\rm ad}_{A})_{s}$ for $A\in V$ and $X\in \n$.
Then we have $[{\rm ad}_{s}(\g), {\rm ad}_{s}(\g)]=0$ and ${\rm ad}_{s}$ is linear (see \cite[Proposition I\hspace{-.1em}I\hspace{-.1em}I.1.1] {DER}).
Since we have $[\g,\g]\subset \n$,  the  map ${\rm ad}_{s}:\g\to D(\g)$ is a representation and the image ${\rm ad}_{s}(\g)$ is abelian and consists of semi-simple elements.
 Let $\bar{\g} ={\rm Im} \,{\rm ad}_{s}\ltimes\g$
and
 \[\frak u=\{X-{\rm ad}_{sX}\in  \bar{\g}  \vert X\in\g\}.\]
Then we have $[\g,\g]\subset \n\subset \frak u$ and $\frak u$ is the nilradical of $\bar \g$ (see \cite{DER}).
Hence we have $\bar \g= {\rm Im} \,{\rm ad}_{s}\ltimes \frak u$.
It is known that the structure of the Lie algebra $\frak u$ is independent of a choice of a subvector space $V$ (see \cite[Corollary I\hspace{-.1em}I\hspace{-.1em}I.3.6]{DER} ).

\begin{lemma}{\rm (\cite[Lemma 2.2]{K2})}\label{semmm}
Suppose $\g=\R^{k}\ltimes _{\phi} \n$ such that $\phi$ is a semi-simple action and $\n$ is nilpotent.
Then the nilshadow $\frak u$ of $\g$ is the direct sum $\R^{k}\oplus \n$.
\end{lemma}

Let $G$ be a simply connected solvable Lie group with the $\R$-Lie algebra $\frak g$.
We denote by ${\rm Ad}_{s}:G\to {\rm Aut}(\g)$ the extension of ${\rm ad}_{s}$.
Then ${\rm Ad}_{s}(G)$ is diagonalizable.
Let $X_{1},\cdots ,X_{n}$ be a basis of $\g\otimes {\C}$ such that ${\rm Ad}_{s}$ is represented by diagonal matrices.
Then we have ${\rm Ad}_{sg}X_{i}=\alpha_{i}(g)X_{i}$ for characters $\alpha_{i}$ of $G$.
Let $x_{1},\dots,x_{n}$ be the dual basis of $X_{1},\dots ,X_{n}$.

We suppose $G$ has a lattice $\Gamma$.
Then we consider the sub-DGA $A^{\ast}_{\Gamma}$ of the de Rham complex $A^{\ast}(G/\Gamma)\otimes \C$ which is given by
\[
A^{p}_{\Gamma}
=\left\langle \alpha_{I} x_{I} {\Big \vert} \begin{array}{cc}I\subset \{1,\dots,n\},\\  (\alpha_{I})_{\vert_{\Gamma}}=1 \end{array}\right\rangle.
\]
where
for a multi-index $I=\{i_{1},\dots ,i_{p}\}$ we write $x_{I}=x_{i_{1}}\wedge\dots \wedge x_{i_{p}}$,  and $\alpha_{I}=\alpha_{i_{1}}\cdots \alpha_{i_{p}}$.
\begin{theorem}{\rm (\cite[Corollary 7.6]{K2})}
Let $G$ be a simply connected solvable Lie group with a lattice $\Gamma$.
Then we have
:\begin{itemize}
\item The inclusion $A^{\ast}_{\Gamma}\subset A^{\ast}(G/\Gamma)\otimes\C$ induces an isomorphism in cohomology.
\item $A^{\ast}_{\Gamma}$ can be regarded as a sub-DGA of $\bigwedge \frak u^{\ast}\otimes \C$.
\end{itemize}
\end{theorem}
We explain the second assertion more precisely.
We consider the subspace $\tilde{\frak u}=\langle \alpha_{1}^{-1}X_{1},\dots, \alpha_{n}^{-1}X_{n}\rangle$ of the space of  complex valued vector fields on $G$.
Then $\tilde{\frak u}=\langle \alpha_{1}^{-1}X_{1},\dots, \alpha_{n}^{-1}X_{n}\rangle$ is a Lie sub-algebra of the Lie algebra of vector fields and the map 
\[\tilde{\frak u}\ni \alpha_{i}^{-1}X_{i}\mapsto X_{i}-{\rm ad}_{sX_{i}}\in \frak u\otimes \C
\]
is a  Lie algebra isomorphism 
where $\frak u$ is the nilshadow of $\g$ (see \cite[Proof of Lemma 5.2]{K2}).

\begin{example}\label{EXC}
Let $\g$ be a $4$-dimensional Lie algebra such that \begin{itemize}
\item $\g=\langle T,X,Y,Z\rangle$
\item $[T,X]=X$, $[T,Y]=-Y$, $[X,Y]=Z$.
\end{itemize}
Then we have the splitting $\g=\langle T\rangle\ltimes \langle X,Y,Z\rangle$ such that $\langle X,Y,Z\rangle$ is the three dimensional real Heisenberg Lie algebra $\frak h(3)$ and the action of $\langle T\rangle$ is semi-simple.
Hence by Lemma \ref{semmm}, the nilshadow $\frak u$ of $\g$ is given by $\frak u=\R\oplus \frak h(3)$.
Hence  as similar to \cite[Example 9.1]{GM1}, the analytic germ $(F(\bigwedge\frak u^{\ast}, \frak a), 0)$ is  equivalent to a cubic cone.

Consider the simply connected solvable Lie group $G$ whose Lie algebra is $\g$.
Then $G$ has a lattice $\Gamma$ \cite{Saw}.
We can easily show that the DGA $A^{\ast}(G/\Gamma)$ is formal and hence the analytic germ $(R(\Gamma, A),{\bf 1})$ at the trivial representation ${\bf 1}$ is    cut out by
 polynomial equations of degree 
at most $2$.
Hence $(R(\Gamma, A),{\bf 1})$ is  linearly embedded in the analytic germ $(F(\bigwedge\frak u^{\ast}, \frak a), 0)$ but its singularity is different from $(F(\bigwedge\frak u^{\ast}, \frak a), 0)$.
\end{example}

By Lemma \ref{semmm}, we give one more corollary of Theorem \ref{realnat}.
\begin{corollary}
Let  $\g=\R^{k}\ltimes _{\phi} \n$ such that $\phi$ is a semi-simple action and $\n$ is a $\nu$-step naturally graded  nilpotent Lie algebra.
Consider the simply connected solvable Lie group $G$ whose Lie algebra is $\g$.
Suppose $G$ has a lattice $\Gamma$.
Then
the analytic germ $(R(\Gamma, A),{\bf 1})$ at the trivial representation ${\bf 1}$   is  cut out by
 polynomial equations of degree 
at most $\nu+1$.
\end{corollary}

\section{Proof of Proposition \ref{KKKUU}}
\subsection{Finite-dimensional DGAs of Poincar\'e duality type}



Let $A^{\ast}$ be a finite-dimensional graded commutative $\C$-algebra.
\begin{definition}[\cite{KSP}]
$A^{\ast}$ is of Poincar\'e duality type (PD-type) if the following conditions hold:
\begin{itemize}
\item $A^{\ast<0}=0$ and $A^{0}=\C 1$ where $1$ is the identity element of $A^{\ast}$.
\item For some positive integer $n$, $A^{\ast>n}=0$ and $A^{n}=\C v$ for $v\not=0$.
\item For any $0<i<n$ the bi-linear map $A^{i}\times A^{n-i}\ni (\alpha,\beta)\mapsto \alpha\cdot \beta\in A^{n}$ is non-degenerate.
\end{itemize}
\end{definition}

Suppose $A^{\ast}$ is of PD-type.
Let $h$ be a Hermitian metric on $A^{\ast}$ which is compatible with the grading.
Take $v\in A^{n}$ such that $h(v,v)=1$.
Define the $\C$-anti-linear map $\bar\ast: A^{i}\to A^{n-i}$ as $\alpha\cdot \bar\ast\beta=h(\alpha,\beta)v$.

\begin{definition}[\cite{KSP}]
A finite-dimensional DGA $(A^{\ast},d)$ is of PD-type if  the following conditions hold:
\begin{itemize}
\item  $A^{\ast}$ is  a finite-dimensional graded $\C$-algebra of PD-type.
\item $dA^{n-1}=0$ and $dA^{0}=0$.
\end{itemize}
\end{definition}

Let $(A^{\ast},d)$ be a finite-dimensional DGA of PD-type.
Denote $d^{\ast}=-\bar\ast d\bar\ast$.
\begin{lemma}[\cite{KSP}]\label{add}
We have $h(d\alpha, \beta)=h(\alpha,d^{\ast}\beta)$ for $\alpha\in A^{i-1}$ and $\beta\in A^{i}$.
\end{lemma}

Define $\Delta=dd^{\ast}+d^{\ast}d$.
and ${\mathcal H}^{\ast}(A)=\ker \Delta$.
By Lemma \ref{add} and finiteness of the dimension of $A^{\ast}$, we can easily show the following lemma.

\begin{lemma}[\cite{KSP}]\label{fiho}
We have the Hodge decomposition
\[A^{r}={\mathcal H}^{r}(A)\oplus \Delta(A^{r})={\mathcal H}^{r}(A)\oplus d(A^{r-1})\oplus d^{\ast}(A^{r+1}).
\]
By this decomposition, the inclusion ${\mathcal H}^{\ast}(A)\subset A^{\ast}$ induces a  isomorphism
\[{\mathcal H}^{p}(A)\cong H^{p}(A)
\]
of vector spaces.
\end{lemma}
We denote by $H$ the projection $H: A^{p} \to{\mathcal H}^{p}(A)$ and define the operator $G$  as the composition
$ \Delta^{-1}_{\vert \Delta(A^{p})}\circ ({\rm id }-H)$.
Let $\beta:A^{\ast}\to dA^{\ast-1}$ be the projection for the decomposition 
\[A^{r}={\mathcal H}^{r}(A)\oplus d(A^{r-1})\oplus d^{\ast}(A^{r+1}).
\]
The restriction map $d: d^{\ast}(A^{\ast})\to d(A^{\ast-1})$ is an isomorphism.
Take the inverse $d^{-1}:d(A^{\ast-1})\to d^{\ast}(A^{\ast})$.
Consider the map $d^{\ast}G:A^{\ast}\to A^{\ast-1}$.
Then for $\omega\in  {\mathcal H}^{r}(A)$, $d^{\ast}x\in d^{\ast}(A^{r})$ and $d^{\ast}y \in d^{\ast}(A^{r+1})$,
we have
\[d^{\ast}G(\omega+dd^{\ast}x+d^{\ast}y)=d^{\ast}(dd^{\ast})^{-1}dd^{\ast}x=d^{\ast}x.
\]
Hence we have $d^{\ast}G=d^{-1}\circ \beta$.

\subsection{Kuranishi spaces of finite-dimensional DGLAs}\label{FINKU}

Let $L^{\ast}$ be a finite-dimensional DGLA with a differential $d$.
Consider the splitting $d(L^{p})\to L^{p}$ for the short exact sequence
\[\xymatrix{
0\ar[r]& {\rm ker} \,d_{\vert_{L^{p}}} \ar[r]& L^{p}\ar[r]^d&d(L^{p})\ar[r]&0
}
\]
and the splitting $H^{p}(L^{\ast})\to  {\rm ker} \,d_{\vert_{L^{p}}}$ for the short exact sequence
\[\xymatrix{
0\ar[r]& d(L^{p-1})\ar[r]&  {\rm ker} \,d_{\vert_{L^{p}}}\ar[r]&H^{p}(L^{\ast})\ar[r]&0.
}
\]
Denote by $\mathcal A^{p}$ and $\mathcal H^{p}$ the images of the splittings $d(L^{p})\to L^{p}$ and $H^{p}(L^{\ast})\to {\rm ker} \,d_{\vert_{L^{p}}}$  respectively.
Then we have 
\[L^{p}= {\mathcal H}^{p}\oplus d(L^{p-1})\oplus \mathcal A^{p}.
\]
Consider the projections  $\beta^{\ast}: L^{\ast}\to  d(L^{\ast-1})$, $H: L^{\ast}\to \mathcal H^{\ast}$ and $\alpha^{\beta}: L^{\ast}\to \mathcal A^{\ast}$.
Since the restriction $ d : \mathcal A^{p} \to d(L^{p})$ is an isomorphism, we have the inverse $d^{-1}: d(L^{p})\to \mathcal A^{p}$ of  $ d : \mathcal A^{p} \to d(L^{p})$.
We define $\delta=d^{-1}\circ \beta:L^{p+1}\to L^{p}$.
Define the map $F: L^{1}\to L^{1}$ as 
\[F(\zeta)=\zeta+\frac{1}{2}\delta[\zeta,\zeta].
\]
Then by the inverse function theorem, on a small ball $B$ in $L^{1}$, the map $F$ is an analytic diffeomorphism.
Then the Kuranishi space ${\mathcal K}(L^{\ast})$ is defined by
\[{\mathcal K}(L^{\ast})=\{\eta\in F(B)\cap \mathcal H^{1}: H([F^{-1}(\eta),F^{-1}(\eta)])=0\}.
\]
It is known that the analytic germ $({\mathcal K}(L^{\ast}),0)$ is equivalent to the germ at the origin for the variety
\[\left\{\zeta\in L^{1}: d\zeta+\frac{1}{2}[\zeta,\zeta]=0, \,\, \delta\zeta=0\right\}
\]
(see \cite[Theorem 2.6]{GM2}).
In particular, if $d(L^{0})=0$, then 
 ${\mathcal K}(L^{\ast})$ is equivalent to the germ at the origin for the variety
\[\left\{\zeta\in L^{1}: d\zeta+\frac{1}{2}[\zeta,\zeta]=0 \right\}.
\]

Take a basis $\zeta_{1},\dots, \zeta_{m}$ of ${\mathcal H}^{1}$.
For parameters $t=( t_{i})$,  we consider the formal power series $\phi(t)=\sum_{r} \phi_{r}(t)$ with values in  $L^{1}$ given inductively by $\phi_{1}(t)=\sum t_{i}\zeta_{j}$ and
\[\phi_{r}(t)=-\frac{1}{2}\sum_{s=1}^{r-1}\delta[\phi_{s}(t), \phi_{r-s}(t)].\]
Then $F^{-1}$ is given by $\phi_{1}(t)\mapsto \phi(t)$ and  
the Kuranishi space $ {\mathcal K}(L^{\ast})$ is an analytic germ in $\C^{m}$ at the origin defined by equations
\[H\left([\phi(t),\phi(t)]\right)=0.
\]

Let $A$ be a finite-dimensional DGA of PD-type and $\g$ a Lie algebra, and
 consider the DGLA $A^{\ast}\otimes \g$.
Then we have the Hodge decomposition
\[A^{\ast}\otimes \g={\mathcal H}^{p}(A)\otimes \g\oplus d(A^{p-1})\otimes \g\oplus d^{\ast}(A^{p+1})\otimes\g
\]
as above with $\delta=d^{\ast} G\otimes {\rm id}$.
Take a basis $\zeta_{1},\dots, \zeta_{m}$ of ${\mathcal H}^{1}(A^{\ast})\otimes \g$.
For parameters $t=( t_{i})$,  we consider the formal power series $\phi(t)=\sum_{r} \phi_{r}(t)$ with values in  $A^{1}\otimes \g$ given inductively by $\phi_{1}(t)=\sum t_{i}\zeta_{j}$ and
\[\phi_{r}(t)=-\frac{1}{2}\sum_{s=1}^{r-1}d^{\ast} G\otimes {\rm id}[\phi_{s}(t), \phi_{r-s}(t)].\]
By the above argument we have the following lemma.
\begin{lemma}\label{kkku}
The analytic  germ $(F(A^{\ast}, \frak g), 0)$ is equivalent to the analytic germ in $\C^{m}$ at the origin defined by equations
\[H\left([\phi(t),\phi(t)]\right)=0.
\]
\end{lemma}

\subsection{Nilpotent Lie algebras}
Let $\n$ be a $\nu$-step nilpotent $K$-Lie algebra for $K=\R$ or $\C$.
Consider the lower central series
\[\n=\n^{(1)}\supset \n^{(2)}\supset\dots \supset \n^{(\nu)}\supset  \n^{(\nu+1)}=\{0\}
\]
where $\n^{(i+1)}=[\n,\n^{(i)}]$.
Take a subspace $\frak a^{(i)}$ such that $\n^{(i)}=\n^{(i+1)}\oplus \frak a^{(i)}$.
We have 
\[\n=\frak a^{(1)}\oplus \frak a^{(2)}\oplus \dots \oplus  \frak a^{(\nu)}.
\]
Consider the dual spaces $ \n^{\ast}$ and $\frak a^{(i)\ast}$ of $\n$ and $\frak a^{(i)}$ respectively.
We consider the cochain complex $\bigwedge \n^{\ast}$ of the Lie algebra with the differential $d$.
Then  $\bigwedge \n^{\ast}$ is a a finite-dimensional DGA of PD-type.
We have 
\[\bigwedge \n^{\ast} =\left(\bigwedge \frak a^{(1)\ast}\right)\wedge \dots \wedge \left(\bigwedge \frak a^{(\nu)\ast}\right).
\]
We have 
\[H^{1}(\n)={\rm ker}\, d_{\bigwedge^{1} \n^{\ast}}=  \frak a^{(1)\ast}.
\]

\begin{lemma}\label{22dd}
\[{\rm ker}\, d_{\bigwedge^{2} \n^{\ast}}\subset \bigoplus_{i+j\le \nu+1, i\le j} \frak a^{(i)\ast}\wedge \frak a^{(j)\ast}.
\]
\end{lemma}
\begin{proof}
Let $\sigma \in {\rm ker}\, d_{\bigwedge^{2} \n^{\ast}}$.
For a positive  integer $k<\nu$,  we say that  $\sigma$ is  $k$-decomposable if we have a decomposition
\[\sigma=\sigma_{1}+\sigma_{2}+\sigma_{3}
\]
such that:
\begin{itemize}
\item 
\[\sigma_{1}\in \bigoplus_{i+j\le \nu+1, i\le j, k<j} \frak a^{(i)\ast}\wedge \frak a^{(j)\ast}.
\]
\item 
\[\sigma_{2}\in \bigoplus_{ i\le k} \frak a^{(i)\ast}\wedge \frak a^{(k)\ast}.
\]
\item
\[\sigma_{3}\in \bigoplus_{ i\le j, j<k} \frak a^{(i)\ast}\wedge \frak a^{(j)\ast}.
\]
\end{itemize}
If $k\le \frac{\nu+1}{2}$, then we have
\[\sigma \in \bigoplus_{i+j\le \nu+1, i\le j} \frak a^{(i)\ast}\wedge \frak a^{(j)\ast}.\]
Consider the case $\frac{\nu+1}{2}<k$.
For $X,Y\in \n$ and $Z\in \n^{(k)}$,
we have $\sigma_{1}([X,Y],Z)=0$, $\sigma_{2}(X,[Y,Z])=0$, $\sigma_{2}(Y,[X,Z])=0$,  $\sigma_{3}([X,Y],Z)=0$, $\sigma_{3}(X,[Y,Z])=0$ and $\sigma_{3}(Y,[X,Z])=0$.
By $d\sigma=0$, we have
\[\sigma_{2}([X,Y],Z)=\sigma_{1}(X,[Y,Z])-\sigma_{1}(Y,[X,Z]).
\] 
Taking $X\in\n $ and $Y\in \n^{(l-1)}$ such that $\nu+1<k+l$, we have
\[\sigma_{2}([X,Y],Z)=0.\]
Hence for $W\in \n^{(l)}$ and $Z\in \n^{(k)}$ such that $\nu+1<k+l$, we have
\[\sigma_{2}(W,Z)=0.\]
Thus we have 
\[\sigma_{2}\in \bigoplus_{i+k\le \nu+1, i\le k} \frak a^{(i)\ast}\wedge \frak a^{(k)\ast}.
\]
Hence taking $\sigma_{1}^{\prime}=\sigma_{1}+\sigma_{2}$ and $\sigma_{3}=\sigma_{2}^{\prime}+\sigma_{3}^{\prime}$ such that 
\[\sigma_{2}^{\prime}\in \bigoplus_{ i\le k-1} \frak a^{(i)\ast}\wedge \frak a^{(k-1)\ast}
\]
and
\[\sigma_{3}^{\prime}\in \bigoplus_{ i\le j, j<k-1} \frak a^{(i)\ast}\wedge \frak a^{(j)\ast},
\]
by the decomposition $\sigma=\sigma_{1}^{\prime}+\sigma^{\prime}_{2}+\sigma^{\prime}_{3}$, $\sigma$ is $(k-1)$-decomposable.
Thus we can say that if $\sigma$ is $k$-decomposable and $\frac{\nu+1}{2}<k-l-1$ for an integer $l$, then $\sigma$ is also $(k-l)$-decomposable.
Take $l$ such that $k-l\le \frac{\nu+1}{2}$.
Then we can say 
\[\sigma \in \bigoplus_{i+j\le \nu+1, i\le j} \frak a^{(i)\ast}\wedge \frak a^{(j)\ast}.\]

Hence it is sufficient to show the above decomposition  of $\sigma$ for $k=\nu-1$.
This was shown in \cite[Lemma 2.8]{BG}.
Hence the Lemma follows.

\end{proof}

It is known that $[\n^{(i)},\n^{(j)}]\subset \n^{(i+j)}$ (see \cite{Cor}) and hence we have
\[d \left(\frak a^{(k)\ast}\right)\subset \bigoplus_{i+j\le k, i\le j} \frak a^{(i)\ast}\wedge \frak a^{(j)\ast}.
\]

\begin{definition}
A nilpotent Lie algebra $\frak n$ is called naturally graded if
we can take subspaces $\frak a^{(i)}\subset \frak n$ such that  $\n^{(i)}=\n^{(i+1)}\oplus \frak a^{(i)}$ and 
$[\frak a^{(i)},\frak a^{(j)}]\subset \frak a^{(i+j)}$ for each $i,j$
where \[\n=\n^{(1)}\supset \n^{(2)}\supset\dots \supset \n^{(\nu)}\supset  \n^{(\nu+1)}=\{0\}
\]
is  the lower central series of $\frak n$.
\end{definition}
If   $\frak n$ is naturally graded, then we have 
\[d \left(\frak a^{(k)\ast}\right)\subset W_{k}
\]
where  $ W_{k}=\bigoplus_{i+j= k, i\le j} \frak a^{(i)\ast}\wedge \frak a^{(j)\ast}$.

Let $g$ be a Hermitian metric on $\n$ such that the sum
\[\n=\frak a^{(1)}\oplus \frak a^{(2)}\oplus \dots \oplus  \frak a^{(\nu)}\]
is an orthogonal direct sum.
Then $g$ give a Hermitian metric on the finite-dimensional DGA  $\bigwedge \n^{\ast}$ of PD-type.
Consider the decomposition
\[\bigwedge^{r}\n^{\ast}={\mathcal H}^{r}(\bigwedge\n^{\ast})\oplus d(\bigwedge^{r-1}\n^{\ast})\oplus d^{\ast}(\bigwedge^{r+1}\n^{\ast}).
\]
Then 
\[\bigwedge^{2}\n^{\ast}=W_{1}\oplus W_{2}\oplus \dots \oplus W_{2\nu}
\]
is an orthogonal direct sum and we have
$d^{-1}\circ \beta(W_{k})\subset \frak a^{(k)\ast}$ by $d(a^{(k)\ast})\subset W_{k}$.

\begin{proposition}

Let $\n$ be a $\nu$-step naturally graded nilpotent Lie algebra and $\frak g$ a Lie algebra.
Then  the analytic germ $(F(\bigwedge \frak u^{\ast}, \frak g), 0)$
 is cut out by polynomial equations of degree at most $\nu+1$.

\end{proposition}
\begin{proof}
Take a basis $\zeta_{1},\dots, \zeta_{m}$ of ${\mathcal H}^{1}(\bigwedge \frak u^{\ast})\otimes \g$.
For parameters $t=( t_{i})$,  we consider the formal power series $\phi(t)=\sum_{r} \phi_{r}(t)$ with values in  $L^{1}$ given inductively by $\phi_{1}(t)=\sum t_{i}\zeta_{j}$ and
\[\phi_{r}(t)=-\frac{1}{2}\sum_{s=1}^{r-1}\delta[\phi_{s}(t), \phi_{r-s}(t)].\]
By Lemma \ref{kkku},  the analytic germ $(F(\bigwedge\frak u^{\ast}, \frak g), 0)$ is equivalent to the analytic germ in $\C^{m}$ at the origine defined by equations
\[ H\left([\phi(t),\phi(t)]\right)=0 
\]
where $H:\bigwedge^{\ast}\n^{\ast}\otimes \g\to {\mathcal H}^{\ast}(\bigwedge \frak u^{\ast})\otimes \g$ is the projection.

We have
\[[ \frak a^{(i)\ast}\otimes \g, \frak a^{(j)\ast}\otimes \g]\subset  W_{i+j}\otimes\g.
\]
By $d^{\ast}G(W_{k})=d^{-1}\circ \beta(W_{k})\subset \frak a^{(k)\ast}$, we have 
\[d^{\ast}G\otimes {\rm id} ([\frak a^{(i)\ast}\otimes \g, \frak a^{(j)\ast}\otimes \g])\subset  \frak a^{(i+j)\ast}\otimes\g .
\]
This implies  $\phi_{r}(t)\in\frak a^{(r)\ast}\otimes  \g $ and
 we have
\[\phi(t)=\phi_{1}(t)+\dots +\phi_{\nu}(t).
\]
By Lemma \ref{22dd}, we have $ {\mathcal H}^{2}(\bigwedge\n^{\ast})\subset {\rm Ker}\, d_{\bigwedge^{2} \n^{\ast}} \subset  \bigoplus_{l\le \nu+1, }W_{l}$ and hence
\[ H(\bigwedge^{2} \n^{\ast}\otimes \g)\subset    {\rm Ker}\, d_{\bigwedge^{2} \n^{\ast}}\otimes\g\subset    \bigoplus_{l\le \nu+1, }W_{l}\otimes\g.\]
Since we have $ [\phi_{i}(t),\phi_{j}(t)]\in W_{i+j}\otimes \g  $ by $\phi_{r}(t)\in \frak a^{(r)\ast}\otimes \g$,
we have $H[\phi_{i}(t),\phi_{j}(t)]=0$ for $\nu+1<i+j$.
Hence $H[\phi(t),\phi(t)]=0$ are polynomial equations of degree 
at most $\nu+1$. 

\end{proof}

\section{Complex analogy}
\subsection{Complex parallelizable solvmanifolds}
Let $G$ be a simply connected $n$-dimensional complex solvable Lie group.
Consider the Lie algebra $\g_{1,0}$ (resp. $\g_{0,1}$) of the left-invariant holomorphic (resp. anti-holomorphic) vector fields on $G$.
Let $N$ be the nilradical  of $G$.
We can take a  simply connected complex nilpotent subgroup $C\subset G$  such that $G=C\cdot N$ (see \cite{dek}).
Since $C$ is nilpotent, the map
\[C\ni c \mapsto ({\rm Ad}_{c})_{s}\in {\rm Aut}(\g_{1,0})\]
is a homomorphism where $({\rm Ad}_{c})_{s}$ is the semi-simple part of ${\rm Ad}_{s}$.

We have a basis $X_{1},\dots,X_{n}$ of $\g_{1,0}$ such that \[({\rm Ad}_{c})_{s}={\rm diag} (\alpha_{1}(c),\dots,\alpha_{n}(c))\] for $c\in C$.
Let $x_{1},\dots, x_{n}$ be the basis of $\g^{\ast}_{1,0}$ which is dual to $X_{1},\dots ,X_{n}$.

\begin{theorem}{\rm(\cite[Corollary 6.2 and Remark 5]{KDD}])}\label{DDDG}
Suppose $G$ has a lattice $\Gamma$.
Let $B^{\ast}_{\Gamma}$ be the subcomplex of $(A^{0,\ast}(G/\Gamma),\bar\partial) $ defined as
\[B^{\ast}_{\Gamma}=\left\langle \frac{\bar\alpha_{I}}{\alpha_{I} }\bar x_{I}{\Big \vert}\left(\frac{\bar\alpha_{I}}{\alpha_{I}}\right)_{ \vert_{\Gamma}}=1\right\rangle
\]
where
for a multi-index $I=\{i_{1},\dots ,i_{p}\}$ we write $x_{I}=x_{i_{1}}\wedge\dots \wedge x_{i_{p}}$,  and $\alpha_{I}=\alpha_{i_{1}}\cdots \alpha_{i_{p}}$.
Consider the nilshadow $\frak u$ of the $\C$-Lie algebra $\g$.
Then we have
:\begin{itemize}
\item The inclusion $B^{\ast}_{\Gamma}\subset A^{0,\ast}(G/\Gamma)$ induces an  isomorphism in cohomology.
\item $B^{\ast}_{\Gamma}$ can be regarded as a sub-DGA of $\bigwedge\frak u^{\ast}$.
\end{itemize}
\end{theorem}
It is known that  a simply connected solvable Lie group $G$ admitting a lattice $\Gamma$ is unimodular.
Hence we have $\alpha_{1}\cdots \alpha_{n}=1$.
For a multi-index $I\subset \{1,\dots, n\}$ and its complement $I^{\prime}=\{1,\dots, n\}-I$,
if $\left(\frac{\bar\alpha_{I}}{\alpha_{I}}\right)_{ \vert_{\Gamma}}=1$
then $\left(\frac{\bar\alpha_{I^{\prime}}}{\alpha_{I^{\prime}}}\right)_{ \vert_{\Gamma}}=1$.
Thus the DGA  $B^{\ast}_{\Gamma}$ as in Theorem \ref{DDDG} is of PD-type.

\subsection{Deformations of holomorphic vector bundles}
 For a compact complex manifold $(M,J)$ and a holomorphic  vector bundle $E$ over $M$,  consider  \[L^{\ast}=A^{0,\ast}(M, {\rm End}(E))\] the differential graded Lie algebra of differential forms of $(0,\ast)$-type with values in the holomorphic vector bundle ${\rm End}(E)$ with the Dolbeault operator induced by the holomorphic structure on $E$.
Then the Kuranishi space
 ${\mathcal K}(L^{\ast})$  represents the  deformation functor for deformations of holomorphic  structures on $E$ (see \cite{GM1} and \cite{GM2}).

Let $G$ be a simply connected $n$-dimensional complex solvable Lie group with a lattice $\Gamma$.
For the Lie algebra ${\frak gl}_{n}(\C)$ of complex valued $n\times n$ matrices,
we consider the DGLA $L^{\ast}=A^{0,\ast}(G/\Gamma)\otimes {\frak gl}_{n}(\C)$.
Then  the Kuranishi space
 ${\mathcal K}(L^{\ast})$  represents the  deformation functor for deformations of holomorphic  structures on $G/\Gamma\times \C^{n}$ near the trivial holomorphic structure.
 As an analytic germ, the Kuranishi space ${\mathcal K}(L^{\ast})$  is an invariant under quasi-isomorphisms between analytic DGLAs.
Hence by Theorem \ref{DDDG}, considering the DGLA $\overline{L^{\ast}}=B^{\ast}_{\Gamma}\otimes  {\frak gl}_{n}(\C)$,  the analytic germ ${\mathcal K}(L^{\ast})$ is equivalent to  ${\mathcal K}(\overline{L^{\ast}})$.
As Section \ref{FINKU}, the analytic germ ${\mathcal K}(\overline{L^{\ast}})$ is equivalent to the analytic germ $(F(B^{\ast}_{\Gamma}, {\frak gl}_{n}(\C)), 0)$.
Hence we have the following theorem.
\begin{theorem}
Let $G$ be a simply connected complex solvable Lie group with a lattice $\Gamma$ and $\g$ the $\C$-Lie algebra of $G$.
We consider the nilshadow $\frak u$ of $\g$.
Then the analytic germ which  represents the  deformation functor for deformations of holomorphic  structures on $G/\Gamma\times \C^{n}$ near the trivial holomorphic structure is  linearly embedded in the analytic germ $(F(\bigwedge\frak u^{\ast}, {\frak gl}_{n}(\C)), 0)$.

Moreover, suppose that the Lie algebra $\frak u$  is $\nu$-step naturally graded.
Then
such analytic germ  is  cut out by
 polynomial equations of degree 
at most $\nu+1$.
\end{theorem}

  \section*{Acknowledgements}
This research is supported by JSPS Research Fellowships for Young Scientists.

\end{document}